\documentclass[12pt, preprint]{amsart}

\usepackage[top=0.9in, bottom=1in, left=1in, right=1in]{geometry}
\usepackage{amsmath, amssymb, amsthm}
\usepackage{verbatim}
\usepackage{graphicx}
\usepackage{enumerate}
\usepackage{mathrsfs}
\usepackage{tcolorbox}
\usepackage{hyperref}
\usepackage{color}
\usepackage{tikz-cd}
\usepackage{csquotes}
\usepackage{hyperref}
\usepackage{bbm}
\usepackage{pdfpages}

\definecolor{darkgreen}{rgb}{0.4,0.0,0.0}

\newtheorem{thm}{Theorem}[section]
\newtheorem{prop}[thm]{Proposition}
\newtheorem{lem}[thm]{Lemma}
\newtheorem{cor}[thm]{Corollary}

\def\XXint#1#2#3{{\setbox0=\hbox{$#1{#2#3}{\int}$ }
\vcenter{\hbox{$#2#3$ }}\kern-.6\wd0}}

\newcommand{\R}{\mathbb{R}}

\newcommand{\C}{\mathbb{C}}

\newcommand{\N}{\mathbb{N}}

\theoremstyle{definition}
\newtheorem{definition}[thm]{Definition}

\newcommand\restr[2]{{
  \left.\kern-\nulldelimiterspace 
  #1 
  \vphantom{\big|} 
  \right|_{#2} 
  }}

\theoremstyle{remark}

\theoremstyle{notation}

\numberwithin{equation}{section}

\newtheoremstyle{ser}
{8pt}
{8pt}
{\it}
{}
{\sf}
{:}
{6mm}
{}

\newtheoremstyle{serr}
{8pt}
{8pt}
{\normalfont}
{}
{\sf}
{.}
{6mm}
{}

\theoremstyle{ser}

\theoremstyle{serr}

\theoremstyle{ser}

\begin{document}

\title{A stronger form of Yamamoto's theorem on singular values}

\author{Soumyashant Nayak}
\address{Statistics and Mathematics Unit\\
Indian Statistical Institute\\
 8th Mile, Mysore Road\\
  RVCE Post, Bengaluru\\
   Karnataka - 560 059, India}
\email{soumyashant@isibang.ac.in}

\maketitle

\begin{abstract}
For a matrix $T \in M_m(\C)$, let $|T| : = \sqrt{T^*T}$. For $A \in M_m(\C)$, we show that the matrix sequence $\big\{ |A^n|^{\frac{1}{n}} \big\}_{n \in \N}$ converges to a positive-semidefinite matrix $H$ whose $j^{\textrm{th}}$-largest eigenvalue is equal to the $j^{\textrm{th}}$-largest eigenvalue-modulus of $A$ (for $1 \le j \le m$). In fact, we give an explicit description of the spectral projections of $H$ in terms of the eigenspaces of the diagonalizable part of $A$ in its Jordan-Chevalley decomposition.  This gives us a stronger form of Yamamoto's theorem which asserts that $\lim_{n \to \infty} s_j(A^n)^{\frac{1}{n}}$ is equal to the $j^{\textrm{th}}$-largest eigenvalue-modulus of $A$, where $s_j(A^n)$ denotes the $j^{\textrm{th}}$-largest singular value of $A^n$. Moreover, we also discuss applications to the asymptotic behaviour of the matrix exponential function, $t \mapsto e^{tA}$.

\bigskip\noindent
{\bf Keywords:}
Singular values, Yamamoto's theorem, spectral-radius formula, matrix exponential function
\vskip 0.01in \noindent
{\bf MSC2010 subject classification:} 15A60, 15A90, 47D06
\end{abstract}

\section{Introduction}

The well-known spectral radius formula for a matrix $A$, $$\rho(A) = \lim_{n \to \infty} \|A^n\|^{\frac{1}{n}},$$ provides insight into the asymptotic behaviour of powers of matrices. Building upon the work of Gautschi (see \cite{gautschi1}, \cite{gautschi2}), Yamamoto considerably refined this result by proving the following theorem.

\vskip 0.1in
\noindent {\bf Yamamoto's theorem } (see \cite[Theorem 1]{yamamoto})
\textsl{Let $A$ be a matrix in $ M_m(\C)$ and $|\lambda_j|(A)$ denote the $j^{\textrm{th}}$-largest number in the list of modulus of eigenvalues of $A$ (counted with multiplicity). Then $$\lim_{n \to \infty} s_j(A^n)^{\frac{1}{n}} = |\lambda_j|(A),$$ where $s_j(A^n)$ denotes the $j^{\textrm{th}}$-largest singular value of $A^n$.
}
\vskip 0.1in
Note that the spectral radius formula corresponds to the case $j = 1$ as $s_1(T) = \|T\|$ for any $T \in M_m(\C)$. In \cite{mathias}, Mathias provides an elegant proof of the above-mentioned result using the interlacing properties of singular values for principal diagonal blocks of a matrix. In \cite{tam-huang}, Tam and Huang generalize the result to the context of real semisimple Lie groups, with the original result corresponding to the case of $SL_n(\C)$.

For an operator $T$ acting on a Hilbert space, we use the notation $|T| := \sqrt{T^*T}$. In this article, our main goal is to prove a stronger form of Yamamoto's theorem by showing the convergence of the matrix sequence $\{ |A^n|^{\frac{1}{n}} \}_{n \in \N}$ for every $m \times m$ complex matrix $A$.
\vskip 0.1in
\noindent {\bf The Main Result } (see Theorem \ref{thm:main})
\textsl{
Let $A \in M_m(\C)$ and $\{ a_1, \ldots, a_k \}$ be the set of modulus of eigenvalues of $A$ such that $0 \le a_1 < a_2 < \cdots < a_k$. Let $A = D + N$ be the Jordan-Chevalley decomposition of $A$ into its commuting diagonalizable and nilpotent parts ($D, N$, respectively). For $1 \le j \le k$, let $E_j$ be the orthogonal projection onto the subspace of $\C ^m$ spanned by the eigenvectors of $D$ corresponding to eigenvalues with modulus less than or equal to $a_j$, and set $E_0 := 0$. Then the following assertions hold:
\begin{itemize}
\item[(i)] The sequence $\{ |A^n|^{\frac{1}{n}} \}_{n \in \N}$ converges to the positive-semidefinite matrix $\sum_{i=1}^k a_j (E_j - E_{j-1}).$ 
\item[(ii)] A non-zero vector $\vec{x} \in \C ^m$ is in $\mathrm{ran}(E_j) \backslash \mathrm{ran}(E_{j-1})$ if and only if $\lim_{n \to \infty} \|A^n \vec{x}\|^{\frac{1}{n}} = a_j.$ 
\item[(iii)] The set $\mathrm{ran}(E_j) \backslash \mathrm{ran}(E_{j-1})$ is invariant under the action of $A^k$ for every $k \in \N$.
\end{itemize}
}
\vskip 0.1in

Since the solution of the system of coupled ordinary differential equations, $$\frac{d\vec{X}(t)}{dt} = A\vec{X}(t), \; \vec{X} : \R \to \C ^m,$$
is given by $\vec{X}(t) = e^{At} \vec{X}(0)$, the asymptotic behaviour of the matrix exponential function, $t \mapsto e^{tA}$, has traditionally been of great interest. In \S \ref{sec:app}, we make some novel observations in this regard (cf.\ \cite[Chapter 4]{batkai_fijavz_rhandi}). Noting that the diagonalizable part of $e^A$ is $e^D$, as a corollary of Theorem \ref{thm:main}, we show that $\lim_{t \to \infty} |e^{tA}|^{\frac{1}{t}}$ exists and provide an explicit description of the spectral projections of the limit (see Theorem \ref{thm:mat_exp}). Furthermore, we show that $\lim_{t \to \infty} \|\vec{X}(t) \|^{\frac{1}{t}}$ exists and compute its value, which provides precise information about the growth of the norm of the solution vector, $\|\vec{X}(t) \|$, as $t \to \infty$. This strengthens Theorem 4.5-(a) in \cite{batkai_fijavz_rhandi}.

Note that the Jordan-Chevalley decomposition of $A$ is an algebraic fact and does {\bf not} care about the inner product on $\C ^m$ where as the adjoint operation is intimately connected with the inner product (and thereby, the Hilbert space structure). Our result shows that the asymptotic behaviour of $\{ |A^n|^{\frac{1}{n}} \}_{n \in \N}$ (and $|e^{tA}|^{\frac{1}{t}}$) is dictated by the algebraic properties of $A$. 

Let $\mathscr{M}$ be a type $II_1$ von Neumann factor. In \cite{haagerup-schultz}, using tools from free probability theory and ultrapower techniques, it was proved by Haagerup and Schultz that the sequence $\{ |T^n|^{\frac{1}{n}} \}_{n \in \N}$ converges in the strong-operator topology (and being a norm-bounded sequence, in the ultra-strong topology). Furthermore, an elementary example (due to Voiculescu, see \cite[Example 8.4]{haagerup-schultz}) is given of a weighted shift operator $S$ on an infinite-dimensional Hilbert space such that the sequence $\{ |S^n|^{\frac{1}{n}}\}_{n \in \N}$ does {\bf not} converge in the strong-operator topology. It is no surprise that the matrix case (which corresponds to finite-dimensional type $I$ factors) affords substantial simplifications and as mentioned above, we are able to obtain precise information about the limit matrix in terms of the diagonalizable part of the Jordan-Chevalley decomposition of $A$. 

\section{Preparatory results}
In this section, we organize some preparatory results on sequences of real numbers, and singular values of matrices, en route to Theorem \ref{thm:main}. First we compile a list of notation used in this article for the reader's quick reference.
\vskip 0.2in

\newpage
\noindent {\bf Notation:} 
\begin{itemize}
\item We use the standard notation $\N, \R, \C$, respectively, to denote the set of natural numbers, real numbers, complex numbers, respectively. 
\item The real part of a complex number $\lambda$ is denoted by $\Re \lambda$. 
\item For a matrix $T \in M_m(\C)$, we denote its range in $\C ^m$ by $\mathrm{ran}(T)$. The multiset of eigenvalues of $T$ is denoted by $\lambda (T)$, and the multiset of modulus of eigenvalues of $T$ is denoted by $|\lambda|(T)$. The $j^{\textrm{th}}$-largest singular value of $T$ is denoted by $s_j (T)$, and the $j^{\textrm{th}}$-largest element of $|\lambda|(T)$ is denoted by $|\lambda _j|(T)$. 

\noindent (A multiset is a collection of objects in which elements may occur more than once but finitely many times, that is, a set with a finite multiplicity function for each of its elements. The underlying set of a multiset is said to be its {\it support}. See \cite{blizard} for a quick introduction to multiset theory.)
\end{itemize}

\subsection{Some elementary results on sequences of real numbers}
\begin{lem}
\label{lem:seq1}
\textsl{
Let $k$ be a fixed positive integer. Then $$\lim_{n \to \infty} \sqrt[n]{\binom{n}{k} } = 1.$$
}
\end{lem}
\begin{proof}
Note that $\lim_{n \to \infty} \sqrt[n]{|n-a|} = 1$ and $\lim_{n \to \infty} \sqrt[n]{a} = 1$ for all $a > 0$. Thus 
 $$\lim_{n \to \infty} \sqrt[n]{\binom{n}{k} }  = \lim_{n \to \infty} \sqrt[n]{n} \sqrt[n]{|n-1|} \cdots \sqrt[n]{|n-k+1|}/\sqrt[n]{k!} = 1.$$
\end{proof}

\begin{lem}
\label{lem:limsup_rad}
\textsl{
Let $\{ a_{1, n} \}_{n \in \N}, \ldots, \{ a_{k, n} \}_{n \in \N}$ be $k$-many sequences of non-negative real numbers and let $b_n := \sum_{i=1}^k a_{i, n}$. Then
$$\limsup_{n} \sqrt[n]{b_n} \le \max_{1 \le i \le k} \big\{ \limsup_n \sqrt[n]{a_{i, n}} \big\}.$$
}
\end{lem}
\begin{proof}
We may assume that $\limsup_n \sqrt[n]{a_{i, n}} < \infty$ for all $1 \le i \le k$, as otherwise there is nothing to prove. Consider the power series $p_i(z) \equiv \sum a_{i, n} z^n$ with radius of convergence, $$R_i = \frac{1}{\limsup_{n} \sqrt[n]{a_{i, n}} } > 0.$$

Clearly, the power series $p(z) \equiv \sum b_n z^n \equiv \sum_{i=1}^k p_i(z)$ converges in the open ball of radius $\min_{1 \le i \le k} \{ R_i \} > 0$, centred at the origin. Thus the radius of convergence of the power series $p$, $$R = \frac{1}{\limsup_n \sqrt[n]{b_n} },$$ is greater than or equal to $\min_{1 \le i \le k} \{ R_i \}$. Taking reciprocals, we get the desired result.
\end{proof}

\begin{lem}
\label{lem:max_comb}
\textsl{
Let $0 \le a_1 <  \ldots < a_m$ and $t_1, t_2, \ldots, t_m \in [0, \infty)$ with $t_m \ne 0$. Then $\lim_{n \to \infty} \big( \sum_{i=1}^m t_i a_i ^n \big)^{\frac{1}{n}} = a_m$.
}
\end{lem}
\begin{proof}
Let $C := \max_{1 \le i \le m} \{ \frac{t_i}{t_m} \} \ge 1$. Note that $t_m a_m ^n \le \sum_{i=1}^m t_i a_i ^n \le m C t_m a_m ^n,$ so that $$a_m(t_m)^{\frac{1}{n}} \le (\sum_{i=1}^m t_i a_i ^n)^{\frac{1}{n}} \le a_m (mCt_m)^{\frac{1}{n}}, \;\;\;\forall n \in \N.$$
The assertion follows from the fact that $\lim_{n \to \infty} t_m ^{\frac{1}{n}} = \lim_{n \to \infty} (mCt_m)^{\frac{1}{n}} = 1$ and the sandwich lemma.
\end{proof}

\subsection{Singular values of matrices}
The usual matrix norm is denoted by $\|\cdot\|$. Since $M_m(\C)$ is a finite-dimensional normed linear space, all norms are equivalent, and the notion of norm-convergence used is immaterial. Most of the results in this subsection follow from standard techniques discussed in the masterful account of the subject of matrix analysis in \cite{bhatia_ma}.
\begin{lem}
\label{lem:sing_conv}
\textsl{
Let $\{ H_n \}_{n \in \N}$ be a sequence of Hermitian matrices in $M_m(\C)$ converging to $H$ in norm. Then $s_j(H_n) \to s_j(H)$, where $s_j(\cdot)$ denotes the $j^{\textrm{th}}$ singular value for $1 \le j \le m$.
}
\end{lem}
\begin{proof}
Applying Weyl's perturbation theorem (see \cite[Theorem VI.2.1]{bhatia_ma}), we have $$|s_j(H_n) - s_j(H) | \le \|H_n - H\|,$$ which proves the assertion.
\end{proof}

\begin{prop}
\label{prop:dom}
\textsl{
Let $A \in M_m(\C)$ with eigenvalues $\lambda_1, \ldots, \lambda_m$ (counted with multiplicity). Then for $0 \le p < \infty$, we have $$\sum_{i=1}^m |\lambda_i|^p \le \mathrm{tr}(|A|^p).$$ 
}
\end{prop}
\begin{proof}
For a positive-semidefinite matrix $H$, the notions of singular value and eigenvalue coincide. Furthermore, we have $s_j(H^p) = s_j(H)^p$ for all $0 < p < \infty$. Thus $\mathrm{tr}(|A|^p) = \sum_{i=1}^m s_i(|A|^p) = \sum_{i=1}^m s_i(A)^p$. The assertion follows from Weyl's majorant theorem (see \cite[Theorem II.3.6]{bhatia_ma}).
\end{proof}

\begin{prop}
\label{prop:holder}
\textsl{
\begin{itemize}
    \item[(i)] (Generalized H\"{o}lder's inequality) Let $A_1, A_2, \ldots, A_k \in M_m(\C)$, and $r, p_1, \cdots, p_k \in (0, \infty)$ be such that $\sum_{i=1}^k \frac{1}{p_i} = \frac{1}{r}$. Then $$\mathrm{tr} \big( |\prod_{i=1}^k A_i|^r \big)^{\frac{1}{r}} \le \prod_{i=1}^k \mathrm{tr}  \big( |A_i|^{p_i} \big)^{\frac{1}{p_i}}.$$
    \item[(ii)] Let $A, B, C \in M_m(\C)$. For every $p \in (0, \infty)$, we have $$\mathrm{tr}(|ABC|^p) \le \|A\|^p \|C\|^p \mathrm{tr}((|B|^p)$$
\end{itemize}
}
\end{prop}
\begin{proof}
Let $A, B \in M_m(\C)$. From (\cite[(III.19)]{bhatia_ma}), we have the following majorization inequality, $$\Big( \log s_1(AB), \ldots, \log s_m(AB) \Big) \prec \Big( \log \big( s_1(A) s_1(B) \big), \ldots, \log \big( s_m(A) s_m(B) \big) \Big).$$
From \cite[Example II.3.5 (v)]{bhatia_ma}, for every $\varphi : [0, \infty) \to \R$ such that $\varphi(e^t)$ is convex and monotone increasing in $t$, we have 
\begin{equation}
\label{ineq:maj_fun}
\sum_{i=1}^m \varphi \big( s_i(AB) \big) \le \sum_{i=1}^m \varphi \big( s_i(A) s_i(B) \big).
\end{equation}

\noindent (i) We prove the result for $k=2$. The general case follows from a standard induction argument. Let $p, q, r \in (0, \infty)$ be such that $\frac{1}{p} + \frac{1}{q} = \frac{1}{r}$. Using inequality (\ref{ineq:maj_fun}) for the function $t \mapsto t^r$, we have 
\begin{align*}
\mathrm{tr}(|AB|^r) = \sum_{i=1}^m s_i(AB)^r &\le \sum_{i=1}^m  s_i(A)^r s_i(B)^r \\
&\le \big( \sum_{i=1}^m s_i(A)^p)^{\frac{r}{p}} \big( \sum_{i=1}^m s_i(B)^q)^{\frac{r}{q}}\\
&= \big( \mathrm{tr}(|A|^p) \big)^{\frac{r}{p}} \big( \mathrm{tr} (|B|^q) \big)^{\frac{r}{q}},
\end{align*}
where the second inequality follows from H\"{o}lder's inequality for non-negative real numbers.
\vskip 0.1in

\noindent (ii) As in part (i), for any two matrices $A, B \in M_m(\C)$ and $p \in (0, \infty)$, we have 
\begin{align*}
\mathrm{tr}(|AB|^p) \le \sum_{i=1}^m  s_i(A)^p s_i(B)^p &\le \sum_{i=1}^m  s_1(A)^p s_i(B)^p \\\
&= \|A\|^p \big( \sum_{i=1}^m s_i(B)^p \big) = \|A\|^p \mathrm{tr}(|B|^p).
\end{align*}
Similarly $\mathrm{tr}(|BA|^p) \le \|A\|^p \mathrm{tr}(|B|^p)$. Thus $$\mathrm{tr}(|ABC|^p) \le \|A\|^p \mathrm{tr}(|BC|^p) \le \|A\|^p\|C\|^p \mathrm{tr}(|B|^p).$$
\end{proof}

\begin{cor}
\label{cor:holder}
\textsl{
Let $A$ be a matrix in $M_m(\C)$. For all $n \in \N$ and $p \in (0, \infty)$, we have $$\mathrm{tr}(|A^n|^{\frac{p}{n}}) \le \mathrm{tr}(|A|^p).$$ 
}
\end{cor}
\begin{proof}
In Proposition \ref{prop:holder}, set $A_1 = \cdots = A_n = A$, $p_1 = \cdots = p_n = p$ and $r = \frac{p}{n}$.
\end{proof}

\begin{lem}
\label{lem:jensen}
\textsl{
Let $A$ be a matrix in $M_m(\C)$ and $\alpha \in [0, 1]$. For every unit vector $\vec{x} \in \C ^m$ and a positive integer $n$, we have $$\| |A^n| ^\alpha \vec{x} \| \le \| A^n \vec{x} \|^\alpha.$$
}
\end{lem}
\begin{proof}
Let $\alpha \in [0, 1]$ and $H$ be a positive-semidefinite matrix in $M_m(\C)$. Since $U^*H^{\alpha} U = (U^*HU)^{\alpha}$ for every unitary matrix $U \in M_m(\C)$, without loss of generality, we may assume that $H = \mathrm{diag}(h_1, \ldots, h_m)$ is in diagonal form. Let $\vec{x} = (x_1, \ldots, x_m) ^{\dagger}$ be a unit vector in $\C ^m$ so that $\sum_{i=1}^m |x_i|^2 = 1$. Since, for $n \in \N$, the function $x \mapsto x^\alpha$ on $[0, \infty)$ is concave, from Jensen's inequality for $h_1, h_2, \ldots, h_m$ with weights $|x_1|^2, \ldots, |x_m|^2$, we see that $$\langle H^\alpha \vec{x}, \vec{x} \rangle = \sum_{i=1}^m |x_i|^2 h_i^\alpha \le \big( \sum_{i=1}^m |x_i|^2 h_i \big)^{\alpha} = \langle H \vec{x}, \vec{x} \rangle^{\alpha}.$$

Using the above inequality for $H^2$, we get 
$$\|H^{\alpha}\vec{x}\|^2 = \langle H^{2 \alpha} \vec{x}, \vec{x} \rangle \le \langle H^2 \vec{x} , \vec{x} \rangle ^{\alpha} = \|H \vec{x} \|^{2 \alpha}$$
which implies that 
\begin{equation}
\label{ineq:jensen}
\|H^{\alpha} \vec{x} \| \le \|H \vec{x}\|^{\alpha}.
\end{equation}

For $T \in M_m(\C)$ and $\vec{x} \in \C ^m$, note that $$\|T \vec{x} \|^2 = \langle T^*T \vec{x}, \vec{x} \rangle = \langle |T|^2 \vec{x}, \vec{x} \rangle = \big\| |T| \vec{x}\big\|^2,$$
which implies that $\| T \vec{x} \| = \big\| |T| \vec{x} \big\|$. We get the desired inequality by plugging in $H = |A^n|$ in inequality (\ref{ineq:jensen}).
\end{proof}

\section{The Main Theorem}

\begin{lem}
\label{lem:diag_conv}
\textsl{
Let $A \in M_m(\C)$ with eigenvalues $\lambda_1, \ldots, \lambda_m$ (counted with multiplicity). Then there is a sequence of invertible matrices $W_n \in GL_n(\C)$ such that $W_n A W_n ^{-1} \to \mathrm{diag}(\lambda_1, \ldots, \lambda_n)$.
}
\end{lem}
\begin{proof}
Without loss of generality, we may assume that $A$ is in upper triangular form (by conjugating with an appropriate unitary). For $n \in \N$, we define $W_n := \mathrm{diag}(1, n, n^2, \ldots, n^{m-1})$. A straightforward computation shows that the $(i, j)^{\mathrm{th}}$ entry of $W_n A W_n ^{-1}$ is $\frac{1}{n^{j-i}}$ times the $(i, j)^{\mathrm{th}}$ entry of $A$ so that the diagonal entries remain unchanged, the superdiagonal entries tend to $0$ as $n \to \infty$ and the subdiagonal entries remain equal to zero.
\end{proof}

\begin{prop}
\label{prop:mod_eig_p_tr}
\textsl{
Let $A \in M_m(\C)$ with eigenvalues $\lambda_1, \ldots, \lambda_m$ (counted with multiplicity). Then for $0 \le p < \infty$, we have $$\sum_{i=1}^m |\lambda_i|^p = \lim_{n \to \infty} \mathrm{tr}(|A^n|^{\frac{p}{n}}).$$
}
\end{prop}
\begin{proof}
From Proposition \ref{prop:dom} and the spectral mapping theorem, for all $n \in \N$ we have $$\sum_{i=1}^m |\lambda_i|^p \le \mathrm{tr}(|A^n|^{\frac{p}{n}}).
$$
Thus
\begin{equation}
\label{eqn:liminf}
\sum_{i=1}^m |\lambda_i|^p \le \liminf_{n \to \infty} \mathrm{tr}(|A^n|^{\frac{p}{n}}).
\end{equation}

Let $W_n$ be as defined in the proof of Lemma \ref{lem:diag_conv}.. For $n \in \N$, define $\Lambda_n := W_n AW_n ^{-1}$. From Proposition \ref{prop:dom}, we have $$\mathrm{tr}(|\Lambda _n|^p) \ge \sum_{i=1}^m |\lambda_i|^p,\;\;\; \forall n \in \N.$$ Since $|\Lambda_n|^p \to \mathrm{diag}(|\lambda_1|^p, \ldots, |\lambda _m|^p)$ as $n \to \infty$, we observe that $$\lim_{n \to \infty} \mathrm{tr}(|\Lambda _n|^p) = \sum_{i=1}^m |\lambda_i|^p.$$

Let $\varepsilon > 0$. Then there exist $k \in \N$ such that $\mathrm{tr}(|\Lambda _k|^p) \le \sum_{i=1}^m |\lambda_i|^p + \varepsilon$. From Corollary \ref{cor:holder}, we see that 
$$\mathrm{tr}(|\Lambda _k ^n|^{\frac{p}{n}}) \le \mathrm{tr}(|\Lambda _k|^p) \le \sum_{i=1}^m |\lambda_i|^p + \varepsilon.$$

Since $A^n = W_k ^{-1} \Lambda_k ^n W_k$, from Proposition \ref{prop:holder}-(ii), it follows that 
\begin{align*}
\mathrm{tr}(|A ^n|^{\frac{p}{n}}) = \mathrm{tr}(|W_k ^{-1} \Lambda_k ^n W_k|^{\frac{p}{n}}) &\le \|W_k ^{-1}\|^{\frac{p}{n}} \|W_k \|^{\frac{p}{n}}\; \mathrm{tr}(|\Lambda _k ^n|^{\frac{p}{n}}) \\
&\le (\|W_k ^{-1}\| \|W_k \|)^{\frac{p}{n}}) \Big( \sum_{i=1}^m |\lambda_i|^p + \varepsilon \Big).
\end{align*}

Thus for all  $\varepsilon >0$, we have 
$$\limsup_{n \to \infty} \mathrm{tr}(|A^n|^{\frac{p}{n}}) \le \sum_{i=1}^m |\lambda_i|^p + \varepsilon,$$
which implies that 
\begin{equation}
\label{eqn:limsup}
\limsup_{n \to \infty} \mathrm{tr}(|A^n|^{\frac{p}{n}}) \le \sum_{i=1}^m |\lambda_i|^p.
\end{equation}
Combining the inequalities (\ref{eqn:liminf}) and (\ref{eqn:limsup}), we get the desired result.
\end{proof}

\begin{prop}
\label{prop:lim_point}
\textsl{
Let $A$ be a matrix in $M_m(\C)$.
\begin{itemize}
\item[(i)] The set of limit points of the sequence $\{ |A^n|^{\frac{1}{n}} \}_{n \in \N}$ is non-empty and consists of positive-semidefinite matrices.
\end{itemize}
Let $H$ be a limit point of the sequence $\{ |A^n|^{\frac{1}{n}} \}_{n \in \N}$.
\begin{itemize}
\item[(ii)] For every $p \in (0, \infty)$, we have $\mathrm{tr}(H^p) = \sum_{i=1}^n |\lambda_i|^p$.
\item[(iii)] $|\lambda|(H) = |\lambda|(A)$.
\end{itemize}
}
\end{prop}
\begin{proof}
\noindent (i) Let $A \in M_m(\C)$. Note that $\||A^n|^{\frac{1}{n}}\| = \||A^n|\|^{\frac{1}{n}} = \|A^n\|^{\frac{1}{n}} \le \|A\|$ for all $n \in \N$. Since the ball of radius $\|A\|$ in $M_m(\C)$ is compact, the set of limit-points of the sequence $\{ |A^n|^{\frac{1}{n}} \}_{n \in \N}$ is non-empty. Positivity of the limit-points follows from the fact that the cone of positive-semidefinite matrices is norm-closed.
\vskip 0.1in

\noindent (ii) Let $\{ |A^{n_k}|^{\frac{1}{n_k}} \}$ be a subsequence of $\{ |A^{n}|^{\frac{1}{n}} \}$ converging to $H$. Then for all $p \in (0, \infty)$, we have $|A^{n_k}|^{\frac{p}{n_k}} \to H^p$ so that $\mathrm{tr}(|A^{n_k}|^{\frac{p}{n_k}} ) \to \mathrm{tr}(H^p)$. Let $\lambda_1, \ldots, \lambda_m$ be the eigenvalues of $A$ (counted with multiplicity). By Proposition \ref{prop:mod_eig_p_tr},  $\mathrm{tr}(H^p) = \sum_{i=1}^m |\lambda _i|^p$ for all $p \in (0, \infty)$.
\vskip 0.1in

\noindent (iii) Let $\mu_1, \ldots, \mu_m$ be the eigenvalues of $H$ (counted with multiplicity). Since $H$ is positive-semidefinite, we have $$\sum_{i=1}^m \mu_i^p = \mathrm{tr}(H^p) = \sum_{i=1}^m |\lambda _i | ^p$$ for all $p \in (0, \infty).$ Thus the multisets $\{ \mu _1, \mu_2,  \ldots, \mu _m \}$ and $\{ |\lambda_1| , |\lambda _2|,  \ldots, |\lambda _m| \}$ are identical.  

\end{proof}

\begin{definition}
Let $A$ be a matrix in $M_m(\C)$. For $r \ge 0$, we define $$V(A, r) := \{ \vec{x} \in \C^m : \limsup_n \|A^n \vec{x}\|^{\frac{1}{n}} \le r\}.$$
\end{definition}

\begin{lem}
\label{lem:D_eig}
\textsl{
Let $A$ be a matrix in $M_m(\C)$. Let $A = D + N$ be the Jordan-Chevalley decomposition of $A$ into its (commuting) diagonalizable and nilpotent parts ($D, N$, respectively). For every $r \ge 0$, the set $V(A, r)$ is a linear subspace of $\C ^m$ and contains the eigenvectors of $D$ corresponding to eigenvalues $\lambda$ with $|\lambda| \le r$.
}
\end{lem}
\begin{proof}
Let $\vec{x}, \vec{y} \in \C ^m$ and $\mu \in \C$. By Lemma \ref{lem:limsup_rad}, we have 
\begin{align*}
\limsup_n \|A^n(\mu \vec{x} + \vec{y})\|^{\frac{1}{n}} &\le \limsup_n \big(|\mu| \|A^n \vec{x}\| + \|A^n\vec{y}\| \big) ^{\frac{1}{n}} \\
&\le \max \big\{ \limsup_n |\mu|^{\frac{1}{n}} \|A^n \vec{x}\|^{\frac{1}{n}}, \limsup_n \|A^n\vec{y}\|^{\frac{1}{n}} \big\} \\
&\le r.
\end{align*}
Thus $\mu \vec{x} + \vec{y} \in V(A, r)$. This shows that $V(A, r)$ is a linear subspace of $\C ^m$.

Let $\vec{x}$ be an eigenvector of $D$ with eigenvalue $\lambda$. Below we show that $\vec{x} \in V(A, |\lambda|)$ (which completes the proof of the lemma). Since $N$ is an $m \times m$ nilpotent matrix, we have $N^m = 0$. Since $D$ and $N$ commute,  for $n  \ge m$ we have $$A^n = (D+N)^n = \sum_{j=0}^{m-1} \binom{n}{j} N^j D^{n-j}.$$ 

Thus $A^n \vec{x} = \sum_{j=0}^{m-1} \binom{n}{j} \lambda^{n-j} N^j \vec{x}.$ From Lemma \ref{lem:seq1} we have $$\lim_n |\lambda|^{1 - \frac{j}{n}} \binom{n}{j}^{\frac{1}{n}} \|N^j \vec{x}\|^{\frac{1}{n}} \le |\lambda|, \textrm{ for } 0 \le j \le m-1.$$

Using Lemma \ref{lem:limsup_rad}, we conclude that
\begin{align*}
\limsup_n \|A^n \vec{x}\|^{\frac{1}{n}} &\le \limsup_n \Big( \sum_{j=0}^{m-1} |\lambda|^{1 - \frac{j}{n}} \binom{n}{j}^{\frac{1}{n}} \|N^j \vec{x}\|^{\frac{1}{n}} \Big)\\
&\le \max_{0 \le j \le m-1} \big\{ \lim_n |\lambda|^{1 - \frac{j}{n}} \binom{n}{j}^{\frac{1}{n}} \|N^j \vec{x}\|^{\frac{1}{n}} \big\} \\
& \le |\lambda|.
\end{align*}
Thus $\vec{x} \in V(A, |\lambda|)$.
\end{proof}

\begin{lem}
\label{lem:pos_inv}
\textsl{
Let $H$ be a positive-semidefinite matrix with spectral decomposition $\sum_{i=1}^k a_i F_i$. Set $a_{k+1} := \infty$. 
\begin{itemize}
\item[(i)] For $1 \le j \le k$ and $r \in [a_j, a_{j+1})$, we have $V(H, r) = V(H, a_j) = \mathrm{ran} \big( \sum_{i=1}^j F_i \big)$;
\item[(ii)] If $H_1, H_2$ are positive-semidefinite matrices in $M_m(\C)$ such that $V(H_1, r) = V(H_2, r)$ for all $r \ge 0$, then $H_1 = H_2$.
\end{itemize}
}
\end{lem}
\begin{proof}
Let $\vec{x}$ be a unit vector in $\C ^m$, and define $t_j := \langle F_j \vec{x} , \vec{x} \rangle \ge 0$. Since $\sum_{i=1}^n t_i = 1$, clearly not all of the $t_i$'s are zero. Let $1 \le \ell \le k$ be the largest integer such that $t_{\ell} \ne 0$. Note that $$\langle H^{2n} \vec{x} , \vec{x} \rangle = \big\langle (\sum_{i=1}^k a_i ^{2n} F_i) \vec{x} , \vec{x} \big\rangle = \sum_{i=1}^m a_i ^{2n} t_i.$$

Using Lemma \ref{lem:max_comb}, we have $$\lim_{n \to \infty} \|H^n \vec{x}\|^{\frac{1}{n}} = \lim_{n \to \infty} \langle H^{2n} \vec{x} , \vec{x} \rangle ^{\frac{1}{2n}} =  a_{\ell}.$$

Thus $\vec{x} \in V(H, a_j)$ if and only if $\sum_{i=j+1}^k \langle F_i \vec{x}, \vec{x} \rangle = 0$, which holds if and only if $\vec{x}$ is in the range of $\sum_{i=1}^j F_i$.
\end{proof}

\begin{prop}
\label{prop:main}
\textsl{
Let A be a matrix in $M_m(C)$. Let $A = D + N$ be the Jordan-Chevalley
decomposition of $A$ into its commuting diagonalizable and nilpotent parts ($D$, $N$, respectively).
Let $H$ be a limit point of the sequence $\{ |A^n|^{\frac{1}{n}} \}_{n \in \N}$. Then for every $r \ge 0$, we have $V(A, r) = V(H,r)$, both of which are equal to the span of the set of eigenvectors of $D$ with eigenvalue-modulus less than or equal to $r$.
}
\end{prop}
\begin{proof}
Let $\{ a_1, a_2, \ldots, a_k \}$ be the support of the multiset of modulus of eigenvalues of $A$, with $0 \le a_1 < a_2 < \cdots < a_k$, and $a_j$ occurring with multiplicity $m_j$. Note that $D$ has the same multiset of eigenvalues as $A$. Since $D$ (being diagonalizable) has a complete set of eigenvectors spanning $\C ^m$, there are $\sum_{i=1}^j m_j$ linearly-independent eigenvectors of $D$ with eigenvalue-modulus less than or equal to $a_j$. Let $r \in [a_j, a_{j+1})$, with the convention $a_{k+1} := \infty$. We observe that 
$$\dim V(H, r) = \dim V(H, a_j)  = \sum_{i=1}^j m_j \le \dim V(A, a_j)$$
where the first equality follows from Lemma \ref{lem:pos_inv}-(i), the second equality follows from Proposition \ref{prop:lim_point}-(iii), and the inequality follows from Lemma \ref{lem:D_eig}. Since $V(A, a_j) \subseteq V(A, r)$, we have \begin{equation}
\label{eq:dim_comp}
 \dim V(H, r) = \sum_{i=1}^j m_j \le \dim V(A, r).
\end{equation}

Let the subsequence $\big\{ |A^{n_k}|^{\frac{1}{n_k}} \big\}$ converge to $H$. Then $\big\{ |A^{n_k}|^\frac{p}{n_k} \big\}$ converges to $H^p$ for every $p \in (0, \infty)$. Using Lemma \ref{lem:jensen}, for every $p \in (0, \infty)$, we have $$\|H^p \vec{x}\| = \lim_{n_k} \big\| |A^{n_k}|^{\frac{p}{n_k}} \vec{x} \big\| \le \limsup_{n_k} \| A^{n_k} \vec{x} \|^{\frac{p}{n_k}} \le \limsup_{n} \|A^n \vec{x}\|^{\frac{p}{n}} \le r^p.$$
Thus $\|H^p \vec{x}\|^{\frac{1}{p}} \le r$ for every $p \in (0, \infty)$ which implies that $\vec{x} \in V(H, r)$. We conclude that 
\begin{equation}
\label{eq:contain}
V(A, r) \subseteq V(H,r), \textrm{ for } 1 \le j \le k.
\end{equation}

Combining (\ref{eq:dim_comp}) and (\ref{eq:contain}), we have $V(A, r) = V(H, r)$ and $\dim V(A, r) = \dim V(H, r) = \sum_{i=1}^{j} m_j$ for $r \in [a_j, a_{j+1})$. In particular, $V(A, r) = V(A, a_j)$ for $r \in [a_j, a_{j+1})$. Thus using Lemma \ref{lem:D_eig}, we conclude that $V(A, r)$ is spanned by the set of eigenvectors of $D$ with eigenvalue-modulus less than or equal to $a_j$.
\end{proof}

\begin{thm}
\label{thm:main}
\textsl{
Let $A \in M_m(\C)$ and $\{ a_1, \ldots, a_k \}$ be the set of modulus of eigenvalues of $A$ such that $0 \le a_1 < a_2 < \cdots < a_k$. Let $A = D + N$ be the Jordan-Chevalley decomposition of $A$ into its commuting diagonalizable and nilpotent parts ($D, N$, respectively). For $1 \le j \le k$, let $E_j$ be the orthogonal projection onto the subspace of $\C ^m$ spanned by the eigenvectors of $D$ corresponding to eigenvalues with modulus less than or equal to $a_j$, and set $E_0 := 0$. Then the following assertions hold:
\begin{itemize}
\item[(i)] The sequence $\{ |A^n|^{\frac{1}{n}} \}_{n \in \N}$ converges to the positive-semidefinite matrix $\sum_{i=1}^k a_j (E_j - E_{j-1}).$ 
\item[(ii)] A non-zero vector $\vec{x} \in \C ^m$ is in $\mathrm{ran}(E_j) \backslash \mathrm{ran}(E_{j-1})$ if and only if $\lim_{n \to \infty} \|A^n \vec{x}\|^{\frac{1}{n}} = a_j.$ 
\item[(iii)] The set $\mathrm{ran}(E_j) \backslash \mathrm{ran}(E_{j-1})$ is invariant under the action of $A^k$ for every $k \in \N$.
\end{itemize}
}
\end{thm}
\begin{proof}
\noindent (i) Let $H_1, H_2$ be limit points of the sequence $\{ |A^n|^{\frac{1}{n}} \}_{n \in \N}$.  By Proposition \ref{prop:main}, $V(A, r) = V(H_1, r) = V(H_2, r)$ for all $r \ge 0$. By Lemma \ref{lem:pos_inv}, $H_1 = H_2$. Hence the sequence $\{ |A^n|^{\frac{1}{n}} \}_{n \in \N}$ converges. The description of $E_j$ also follows from Proposition \ref{prop:main} and Lemma \ref{lem:pos_inv}.
\vskip 0.15in

\noindent (ii) For $\vec{y} \in \C^m$ if $\lim_{n \to \infty} \|A^n \vec{y} \|^{\frac{1}{n}} = a_j$, from the definition of $V(A, a_j)$ and $V(A, a_{j-1})$, it follows that $\vec{y} \in V(A, a_j) \backslash V(A, a_{j-1})$.

The converse needs a bit more work. Define $H_k := |A^k|^{\frac{1}{k}}$ for $k \in \N$. From part (i), $H_n \to H$ as $n \to \infty$ for some positive-semidefinite matrix in $M_m(\C)$. Thus $H_n ^m \to H^m$ as $n \to \infty$ for every $m \in \N$.  By Lemma \ref{lem:jensen} (considering $\alpha = \frac{m}{n}$ with sufficiently large $n$ so that $\alpha \in [0, 1]$), for every $m \in \N$ and $\vec{y} \in \C^m$,  we have $$\|H^m \vec{y} \|^{\frac{1}{m}} = \lim_{n \to \infty} \|H_n ^m \vec{y}\|^{\frac{1}{m}} \le \liminf_{n \to \infty} \|A^n \vec{y} \|^{\frac{1}{n}} \le \limsup_{n \to \infty} \|A^n \vec{y} \|^{\frac{1}{n}}.$$

By Proposition \ref{prop:main}, $\mathrm{ran}(E_j) = V(A, a_j) = V(H, a_j)$. If $\vec{x} \in V(A, a_j) \backslash V(A, a_{j-1})$, note that 
$$a_j = \lim_{m \to \infty} \|H^m \vec{x} \|^{\frac{1}{m}} \le \liminf_{n \to \infty} \|A^n \vec{x} \|^{\frac{1}{n}} \le \limsup_{n \to \infty} \|A^n \vec{y} \|^{\frac{1}{n}} \le a_j.$$
Thus $\lim_{n \to \infty} \|A^n \vec{x} \|^{\frac{1}{n}} = a_j$. 
\vskip 0.15in

\noindent (iii) Note that for fixed $k \in \N$, we have $$\lim_{n \to \infty} \|A^n \vec{x} \|^{\frac{1}{n}} = \lim_{n \to \infty} \|A^{n+k} \vec{x}\|^{\frac{1}{n+k}} = \lim_{n \to \infty} \|A^{n+k} \vec{x}\|^{\frac{1}{n}} = \lim_{n \to \infty} \|A^{n} (A^k\vec{x})\|^{\frac{1}{n}}.$$
The assertion follows from part (ii).
\end{proof}

\begin{cor}[Yamamoto's theorem]
\textsl{
Let $A \in M_m(\C)$ with its $j^{\mathrm{th}}$-largest singular value denoted by $s_j(A)$, and $j^{\mathrm{th}}$ largest eigenvalue-modulus denoted by $|\lambda_j|(A)$. Then for all $1 \le j \le m$, we have $$\lim_{n \to \infty} s_j(A^n)^{\frac{1}{n}} = |\lambda _j|(A).$$
}
\end{cor}
\begin{proof}
Note that $s_j(A^n)^{\frac{1}{n}} = s_j(|A^n|^{\frac{1}{n}})$. Let $H = \lim_{n \to \infty} |A^n|^{\frac{1}{n}}$ (this limit exists by Theorem \ref{thm:main}). From Lemma \ref{lem:sing_conv} and Proposition \ref{prop:lim_point}-(iii), we conclude that $$\lim_{n \to \infty} s_j(A^n)^{\frac{1}{n}} = s_j(H) = |\lambda _j|(H) = |\lambda _j|(A).$$
\end{proof}

\section{Applications to linear systems of ordinary differential equations}

\label{sec:app}

In this section, we discuss the insights provided by Theorem \ref{thm:main} to the study of linear systems of ordinary differential equations with constant coefficients. We consistently use the notation from Theorem \ref{thm:mat_exp} below throughout this section.

\begin{thm}
\label{thm:mat_exp}
\textsl{
Let $A \in M_m(\C)$ and $\{ h_1, \ldots, h_k \}$ be the set of real-parts of eigenvalues of $A$ such that $h_1 < h_2 < \cdots < h_k$. Let $A = D + N$ be the Jordan-Chevalley decomposition of $A$ into its commuting diagonalizable and nilpotent parts ($D, N$, respectively). For $1 \le j \le k$, let $F_j$ be the orthogonal projection onto the subspace of $\C ^m$ spanned by the eigenvectors of $D$ corresponding to eigenvalues with real-part less than or equal to $h_j$, and set $E_0 := 0$. Then the following assertions hold:
\begin{itemize}
\item[(i)] $\lim_{t \to \infty} |e^{tA}|^{\frac{1}{t}} = \sum_{i=1}^k e^{h_j} (F_j - F_{j-1}).$
\item[(ii)] A non-zero vector $\vec{x} \in \C^m$ is in $\mathrm{ran}(F_j) \backslash \mathrm{ran}(F_{j-1})$ if and only if $\lim_{t \to \infty} \|e^{tA} \vec{x}\|^{\frac{1}{t}} = e^{h_j}.$ 
\item[(iii)] The set $\mathrm{ran}(F_j) \backslash \mathrm{ran}(F_{j-1})$ is invariant under the action of $e^{sA}$ for every $s \ge 0$.
\end{itemize}
}
\end{thm}
\begin{proof}
We note three pertinent observations regarding exponentials of matrices and complex numbers.
\begin{itemize}
    \item[(a)] The Jordan-Chevalley decomposition of $e^A$ is given by $e^A = e^D + e^D(e^N-I)$ so that $e^D$ is the diagonalizable part of $e^A$.
    \item[(b)] A vector $\vec{x}$ is an eigenvector of $D$ with eigenvalue $\lambda$ if and only if $\vec{x}$ is an eigenvector of $e^D$ with eigenvalue $e^{\lambda}$.
    \item[(c)] For $\lambda _1, \lambda _2 \in \C$, we have $|e^{\lambda_1}| < |e^{\lambda_2}|$ ($|e^{\lambda_1}| = |e^{\lambda_2}|$, respectively) if and only if $\Re \lambda_1 < \Re \lambda_2$ ($\Re \lambda_1 = \Re \lambda_2$, respectively)
\end{itemize}

Applying Theorem \ref{thm:main} to the matrix $e^A$, we conclude that $$\lim_{n \to \infty} |e^{nA} |^{\frac{1}{n}} = \sum_{i=1}^k e^{h_j} (F_j - F_{j-1}),$$
and a vector $\vec{x} \in \C ^m$ is in $\mathrm{ran}(F_j) \backslash \mathrm{ran}(F_{j-1})$ if and only if $\lim_{n \to \infty} \|e^{nA} \vec{x}\|^{\frac{1}{n}} = e^{h_j}.$
What remains to be proved is that the limits in (i) and (ii) exist as $t \to \infty$ in $\R$ and are equal to their respective limits as $n \to \infty$ in $\N$.

Note that the function $t \mapsto \|e^{tA}\|$ is continuous and only takes strictly positive values as $e^{tA}$ is invertible for any $t \in \R$. Let $c := \min_{t \in [0, 1] } \|e^{-tA}\|^{-1}$ and $C := \max_{t \in [0, 1]} \|e^{tA}\|$. Clearly $0 < c \le C$. For $\alpha \in [0, 1)$, we have $c^2 I \le (e^{\alpha A})^* e^{\alpha A} \le C^2 I$. For every positive integer $n$, we note that 
\begin{equation}
\label{eqn:exp}
c^2 (e^{n A})^* e^{n A} \le (e^{(n + \alpha) A})^* e^{(n + \alpha) A} \le C^2 (e^{n A})^* e^{n A}
\end{equation}

As $x \mapsto x^{\frac{1}{2n}}$ is an operator-monotone function on $[0, \infty)$, we observe that $$c^{\frac{1}{n}} | e^{nA}|^{\frac{1}{n}} \le |e^{(n+\alpha)A}|^{\frac{1}{n}} \le C^{\frac{1}{n}} | e^{nA}|^{\frac{1}{n}}.$$
Since $\lim_{n \to \infty} c^{\frac{1}{n}} = \lim_{n \to \infty} C^{\frac{1}{n}} = 1$ and $\lim_{t \to \infty} \frac{\lfloor t \rfloor}{t} = 1$, we conclude that $$\lim_{t \to \infty} |e^{tA}|^{\frac{1}{t}} = \lim_{t \to \infty} |e^{tA}|^{\frac{1}{\lfloor t \rfloor}} = \lim_{n \to \infty} |e^{nA}|^{\frac{1}{n}}.$$
For every $\vec{x} \in \C^m$, from inequality (\ref{eqn:exp}) we have that
$$c^{\frac{1}{n}} \|e^{nA} \vec{x}\| ^{\frac{1}{n}} \le \|e^{(n+\alpha)A} \vec{x}\|^{\frac{1}{n}} \le C^{\frac{1}{n}} \|e^{nA} \vec{x}\| ^{\frac{1}{n}}.$$
Thus $$\lim_{t \to \infty} \| e^{tA} \vec{x} \|^{\frac{1}{t}} = \lim_{t \to \infty} \| e^{tA} \|^{\frac{1}{\lfloor t \rfloor}} = \lim_{n \to \infty} \| e^{nA} \vec{x} \|^{\frac{1}{n}}.$$

Part (iii) follows from part (ii) together with the fact that for fixed $s \ge 0$, we have $$\lim_{t \to \infty} \|e^{tA} \vec{x} \|^{\frac{1}{t}} = \lim_{t \to \infty} \|e^{(t+s)A} \vec{x}\|^{\frac{1}{t+s}} = \lim_{t \to \infty} \|e^{(t+s)A} \vec{x}\|^{\frac{1}{t}} = \lim_{t \to \infty} \|e^{tA} (e^{sA} \vec{x})\|^{\frac{1}{t}}.$$
\end{proof}

Consider the linear homogeneous system of differential equations with constant coefficients, 
\begin{equation}
\label{eqn:lin_sys}
\frac{dx_i}{dt} = \sum_{j=1}^m a_{ij}x_j, \;\; i= 1, 2, \ldots, m,
\end{equation}
that is, $$\frac{d\vec{X}}{dt} = A \vec{X},$$
where $\vec{X} = (x_1, \ldots, x_m)^{T} : \R \to \C ^m$. The unique solution to the above system is given by $\vec{X}(t) = e^{tA} \vec{X}(0)$, where $\vec{X}(0)$ denotes the vector of initial conditions (see \cite[Corollary 4.3]{batkai_fijavz_rhandi}). 
\vskip 0.2in

\noindent {\bf Observation 1:}
\textsl{
Let $\lambda$ be an eigenvalue of $A$ so that $\Re \lambda = h_j$ for some $1 \le j \le k$. From Theorem \ref{thm:mat_exp}-(ii), it follows that if $\vec{X}(0) \in \mathrm{ran}(F_j) \backslash \mathrm{ran}(F_{j-1})$, there exist $M \ge 1$ and $N > 0$ such that $$N e^{\rho t }\|\vec{X}(0)\| \le \|\vec{X}(t) \| \le Me^{\omega t} \|\vec{X}(0)\|,$$ for all $t \ge 0$, and $\rho < h_j < \omega$.
}
\vskip 0.1in
This gives us a slightly stronger version of \cite[Theorem 4.5(a)]{batkai_fijavz_rhandi} which provides the above bounds only in the cases where $\vec{X}(0)$ is an eigenvector of $D$ with eigenvalue $\lambda$. (Note that the $\lambda$-eigenspace of $D$ is contained in $\mathrm{ran}(F_j) \backslash \mathrm{ran}(F_{j-1})$).
\vskip 0.2in
\noindent {\bf Observation 2:} 
\textsl{
If the vector of initial conditions, $\vec{X}(0)$, is in $\mathrm{ran}(F_j) \backslash \mathrm{ran}(F_{j-1})$, then from Theorem \ref{thm:mat_exp}-(iii), it follows that $\vec{X}(t)$ is in $\mathrm{ran}(F_j) \backslash \mathrm{ran}(F_{j-1})$ for all $t \ge 0$. 
}
\vskip 0.1in

In other words, the above observation tells us that if the vector of initial conditions, $\vec{X}(0)$, is {\bf not} in the region $\Omega :=  \mathrm{ran}(F_{j-1}) \cup \big( \C ^m \backslash \mathrm{ran}(F_j) \big) $, then $\vec{X}(t)$ avoids the region $\Omega$ for all $t \ge 0$.

\section{Acknowledgements}
This work is supported by the Startup Research Grant (SRG/2021/002383) of SERB (Science and Engineering Research Board, Govt.\ of India). I would like to express my gratitude to B.\ V.\ Rajarama Bhat for discussions and suggestions that helped improve the presentation in this article.

\bibliographystyle{plain}
\bibliography{references}

\end{document}